\documentclass[12pt]{amsart}
\usepackage{amsmath,amsthm,amsfonts,amssymb,mathrsfs}
\date{\today}

\input xymatrix
\xyoption{all}

\usepackage{color}

\newtheorem{theorem}{Theorem}
\newtheorem{proposition}[theorem]{Proposition}
\newtheorem{lemma}[theorem]{Lemma}

\theoremstyle{definition}
\newtheorem{remark}[theorem]{Remark}

\usepackage{hyperref}

\begin{document}

\title[Congruences on the monoid of monotone injective partial
selfmaps of $L_n\times_{\operatorname{lex}}\mathbb{Z}$]{Congruences on the monoid of monotone injective partial selfmaps of $L_n\times_{\operatorname{lex}}\mathbb{Z}$ with co-finite domains and images}

\author[O.~Gutik and I.~Pozdniakova]{Oleg~Gutik and Inna Pozdniakova}
\address{Faculty of Mechanics and Mathematics,
National University of Lviv, Universytetska 1, Lviv, 79000, Ukraine}
\email{o\_gutik@franko.lviv.ua, ovgutik@yahoo.com, pozdnyakova.inna@gmail.com}



\begin{abstract}
We study congruences on the semigroup $\mathscr{I\!O}\!_{\infty}(\mathbb{Z}^n_{\operatorname{lex}})$ of monotone injective partial selfmaps of the set of $L_n\times_{\operatorname{lex}}\mathbb{Z}$ having co-finite domains and images, where $L_n\times_{\operatorname{lex}}\mathbb{Z}$ is the lexicographic product of $n$-elements chain and the set of integers with the usual linear order.  The structure of the sublattice of congruences on $\mathscr{I\!O}\!_{\infty}(\mathbb{Z}^n_{\operatorname{lex}})$ which contained in the least group congruence is described.
\end{abstract}

\maketitle



We follow the terminology of \cite{Howie1995,Lawson1998} and \cite{Petrich1984}. We shall denote the additive group of integers by $\mathbb{Z}(+)$.

An algebraic semigroup $S$ is called {\it inverse} if for any
element $x\in S$ there exists a unique $x^{-1}\in S$ such that
$xx^{-1}x=x$ and $x^{-1}xx^{-1}=x^{-1}$. The element $x^{-1}$ is
called the {\it inverse of} $x\in S$. If $S$ is an inverse
semigroup, then the function $\operatorname{inv}\colon S\to S$ which
assigns to every element $x$ of $S$ its inverse element $x^{-1}$ is
called an {\it inversion}.

If $\mathfrak{C}$ is an arbitrary congruence on a semigroup $S$,
then we denote by $\Phi_\mathfrak{C}\colon S\rightarrow
S/\mathfrak{C}$ the natural homomorphisms from $S$ onto the quotient
semigroup $S/\mathfrak{C}$. A congruence $\mathfrak{C}$ on a
semigroup $S$ is called \emph{non-trivial} if $\mathfrak{C}$ is
distinct from universal and identity congruences $\Delta_S$ on $S$, and
\emph{group} if the quotient semigroup $S/\mathfrak{C}$ is a group.
Every inverse semigroup $S$ admits the least (minimum) group
congruence $\sigma$:
\begin{equation*}
    a\sigma b \; \hbox{ if and only if there exists }\;
    e\in E(S) \; \hbox{ such that }\; ae=be
\end{equation*}
(see \cite[Lemma~III.5.2]{Petrich1984})

If $S$ is a semigroup, then we shall denote the subset of idempotents of $S$ by $E(S)$. If $S$ is an inverse semigroup, then $E(S)$ is closed under multiplication and we shall refer to $E(S)$ as a \emph{band} (or the \emph{band of} $S$). If the band $E(S)$ is a non-empty subset of $S$, then the semigroup operation on $S$ determines the following partial order $\leqslant$ on $E(S)$: $e\leqslant f$ if and only if $ef=fe=e$. This order is called the {\em natural partial order} on $E(S)$. A \emph{semilattice} is a commutative semigroup of idempotents. A semilattice $E$ is called {\em linearly ordered} or a \emph{chain} if its natural order is a linear order. A \emph{maximal chain} of a semilattice $E$ is a chain which is properly contained in no other chain of $E$.

If $S$ is a semigroup, then we shall denote the Green relations on
$S$ by $\mathscr{R}$, $\mathscr{L}$, $\mathscr{J}$, $\mathscr{D}$
and $\mathscr{H}$ (see \cite[Section~2.1]{CP}):
\begin{align*}
    &\qquad a\mathscr{R}b \mbox{ if and only if } aS^1=bS^1;\\
    &\qquad a\mathscr{L}b \mbox{ if and only if } S^1a=S^1b;\\
    &\qquad a\mathscr{J}b \mbox{ if and only if } S^1aS^1=S^1bS^1;\\
    &\qquad \mathscr{D}=\mathscr{L}\circ\mathscr{R}=
          \mathscr{R}\circ\mathscr{L};\\
    &\qquad \mathscr{H}=\mathscr{L}\cap\mathscr{R}.
\end{align*}
A semigroup $S$ is called \emph{simple} if $S$ contains no proper two-sided ideal, i.e., $S$ has a unique $\mathscr{J}$-class, and \emph{bisimple} if $S$ has a unique $\mathscr{D}$-class.

If $\alpha\colon X\rightharpoonup Y$ is a partial map, then by $\operatorname{dom}\alpha$ and $\operatorname{ran}\alpha$ we denote the domain and the range of $\alpha$, respectively.

Let $\mathscr{I}_\lambda$ denotes the set of all partial one-to-one transformations of an infinite set $X$ of cardinality $\lambda$ endowed with the following semigroup operation: $x(\alpha\beta)=(x\alpha)\beta$ if $x\in\operatorname{dom}(\alpha\beta)=\{ y\in\operatorname{dom}\alpha\mid y\alpha\in\operatorname{dom}\beta\}$, for
$\alpha,\beta\in\mathscr{I}_\lambda$. The semigroup $\mathscr{I}_\lambda$ is called the \emph{symmetric inverse semigroup} over the set $X$~(see \cite[Section~1.9]{CP}). The symmetric inverse semigroup was introduced by Wagner~\cite{Wagner1952} and it plays a major role in the theory of semigroups. An element  $\alpha\in\mathscr{I}_\lambda$ is called \emph{cofinite}, if the sets $\lambda\setminus\operatorname{dom}\alpha$ and $\lambda\setminus\operatorname{ran}\alpha$ are finite.

Let $(X,\leqslant)$ be a partially ordered set. We shall say that a partial map
$\alpha\colon X\rightharpoonup X$ is \emph{monotone} if $x\leqslant y$ implies $(x)\alpha\leqslant(y)\alpha$ for each  $x,y\in X$.

Let $\mathbb{Z}$ be the set of integers with the usual linear order ``$\le$''. For any positive integer $n$ by $L_n$ we denote the set $\{1,\ldots,n\}$ with the usual linear order ``$\le$''. On the Cartesian product $L_n\times\mathbb{Z}$ we define the lexicographic order, i.e.,
\begin{equation*}
    (i,m)\leqslant(j,n) \qquad \hbox{if and only if} \qquad (i<j) \quad \hbox{or} \quad (i=j \quad\hbox{and} \quad m\le n).
\end{equation*}
Later the set $L_n\times\mathbb{Z}$ with the lexicographic order we denote by $L_n\times_{\operatorname{lex}}\mathbb{Z}$. Also, it is obvious that the set $\mathbb{Z}\times L_n$ with the lexicographic order is order isomorphic to $(\mathbb{Z},\le)$.

By $\mathscr{I\!O}\!_{\infty}(\mathbb{Z}^n_{\operatorname{lex}})$ we denote a semigroup of injective partial monotone selfmaps of $L_n\times_{\operatorname{lex}}\mathbb{Z}$ with co-finite domains and images.
Obviously, $\mathscr{I\!O}\!_{\infty}(\mathbb{Z}^n_{\operatorname{lex}})$ is an
inverse submonoid of the semigroup $\mathscr{I}_\omega$ and $\mathscr{I\!O}\!_{\infty}(\mathbb{Z}^n_{\operatorname{lex}})$ is a countable semigroup. Also, by $\mathscr{I\!O}\!_{\infty}(\mathbb{Z})$ we denote a semigroup of injective partial monotone selfmaps of $\mathbb{Z}$ with cofinite domains and images.

Furthermore, we shall denote the identity of the semigroup
$\mathscr{I\!O}\!_{\infty}(\mathbb{Z}^n_{\operatorname{lex}})$ by $\mathbb{I}$ and
the group of units of
$\mathscr{I\!O}\!_{\infty}(\mathbb{Z}^n_{\operatorname{lex}})$ by $H(\mathbb{I})$.

Gutik and Repov\v{s} in \cite{GutikRepovs2011} showed that the semigroup $\mathscr{I}_{\infty}^{\!\nearrow}(\mathbb{N})$ of partial cofinite monotone injective transformations of the set of positive integers $\mathbb{N}$ has algebraic properties similar to those of the bicyclic semigroup: it is bisimple and all of its non-trivial semigroup homomorphisms are either isomorphisms or group
homomorphisms.

In \cite{GutikRepovs2012} Gutik and Repov\v{s} studied the semigroup $\mathscr{I}_{\infty}^{\!\nearrow}(\mathbb{Z})$ of partial cofinite monotone injective transformations of the set of integers $\mathbb{Z}$ and they showed that
$\mathscr{I}^{\!\nearrow}_{\infty}(\mathbb{Z})$ is bisimple and all of its non-trivial semigroup homomorphisms are either isomorphisms or group homomorphisms.

In the paper \cite{GutikPozdnyakova2014} we studied the semigroup
$\mathscr{I\!O}\!_{\infty}(\mathbb{Z}^n_{\operatorname{lex}})$. There we described
Green's relations on $\mathscr{I\!O}\!_{\infty}(\mathbb{Z}^n_{\operatorname{lex}})$,
showed that the semigroup $\mathscr{I\!O}\!_{\infty}(\mathbb{Z}^n_{\operatorname{lex}})$ is
bisimple and established its projective congruences. Also, there we proved that $\mathscr{I\!O}\!_{\infty}(\mathbb{Z}^n_{\operatorname{lex}})$ is finitely generated, every automorphism of $\mathscr{I\!O}\!_{\infty}(\mathbb{Z})$ is inner and showed that in the case $n\geqslant 2$ the semigroup $\mathscr{I\!O}\!_{\infty}(\mathbb{Z}^n_{\operatorname{lex}})$ has non-inner automorphisms. In \cite{GutikPozdnyakova2014} we proved that for every positive integer $n$ the quotient semigroup
$\mathscr{I\!O}\!_{\infty}(\mathbb{Z}^n_{\operatorname{lex}})/\sigma$, where $\sigma$ is the least group congruence on $\mathscr{I\!O}\!_{\infty}(\mathbb{Z}^n_{\operatorname{lex}})$, is isomorphic to the direct power $\left(\mathbb{Z}(+)\right)^{2n}$.


By Proposition~2.3$(iv)$ \cite{GutikPozdnyakova2014}, the semigroup $\mathscr{I\!O}\!_{\infty}(\mathbb{Z}^n_{\operatorname{lex}})$ is isomorphic to the direct power $\left(\mathscr{I\!O}\!_{\infty}(\mathbb{Z})\right)^{n}$. Fixing this isomorphism further we shall identify elements of the semigroup $\mathscr{I\!O}\!_{\infty}(\mathbb{Z}^n_{\operatorname{lex}})$ with elements of the direct product $\left(\mathscr{I\!O}\!_{\infty}(\mathbb{Z})\right)^{n}$, i.e., every element $\alpha$ of $\mathscr{I\!O}\!_{\infty}(\mathbb{Z}^n_{\operatorname{lex}})$ we present in the form $(\alpha_1,\alpha_2,\ldots,\alpha_n)$, where
all $\alpha_i$ belongs to $\mathscr{I\!O}\!_{\infty}(\mathbb{Z})$. Later by $\alpha_i^\circ$ we shall denote the element with the form $(\mathbb{I}_1,\ldots,\mathbb{I}_{i-1},\alpha_i,\mathbb{I}_{i+1},\ldots,\mathbb{I}_n)$, where $\mathbb{I}_{j}$ is the identity of the $j$-th factor of $\left(\mathscr{I\!O}\!_{\infty}(\mathbb{Z})\right)^{n}$
for all $j$ and $\alpha_i\in \left(\mathscr{I\!O}\!_{\infty}(\mathbb{Z})\right)$.
It is obvious that for every $\alpha=(\alpha_1,\ldots,\alpha_n)\in \mathscr{I\!O}\!_{\infty}(\mathbb{Z}^n_{\operatorname{lex}})$ we have that $\alpha=\alpha_1^\circ\ldots\alpha_n^\circ$.

For every $i=1,\ldots,n$ we define a binary relation $\sigma_{[i]}$ on the semigroup $\mathscr{I\!O}\!_{\infty}(\mathbb{Z}^n_{\operatorname{lex}})$ in the following way:
\begin{equation*}
    \alpha\sigma_{[i]}\beta \quad \hbox{ if and only if there exists an idempotent } \varepsilon\in\mathscr{I\!O}\!_{\infty}(\mathbb{Z}^n_{\operatorname{lex}}) \hbox{ such that } \alpha\varepsilon_i^\circ=\beta\varepsilon_i^\circ.
\end{equation*}

In \cite{GutikPozdnyakova2014} we proved that $\sigma_{[i]}$ is a congruence on $\mathscr{I\!O}\!_{\infty}(\mathbb{Z}^n_{\operatorname{lex}})$ for every $i=1,\ldots,n$. Also, there is shown that for any subset $\{i_1,\ldots,i_k\}\subseteq\{1,\ldots,n\}$ of distinct integers, the relation $\sigma_{[i_1,\ldots,i_k]}= \sigma_{[i_1]}\circ\ldots\circ\sigma_{[i_k]}$ is a congruence on $\mathscr{I\!O}\!_{\infty}(\mathbb{Z}^n_{\operatorname{lex}})$ and is described the properties of the congruence $\sigma_{[i_1,\ldots,i_k]}$ (see Propositions~2.11-2.13, 2.15 and 2.18 in \cite{GutikPozdnyakova2014}). Moreover,
$\sigma_{[1,2,\ldots,n]}$ is the least group congruence on the semigroup $\mathscr{I\!O}\!_{\infty}(\mathbb{Z}^n_{\operatorname{lex}})$.

For every $i=1,\ldots,n$ we define a map $\pi^i\colon \mathscr{I\!O}\!_{\infty}(\mathbb{Z}^n_{\operatorname{lex}})\rightarrow \mathscr{I\!O}\!_{\infty}(\mathbb{Z}^n_{\operatorname{lex}})$ by the formula $(\alpha)\pi^i=\alpha_i^\circ$, i.e., $(\alpha_1,\ldots,\alpha_i,\ldots,\alpha_n)\pi^i= (\mathbb{I}_1,\ldots,\mathbb{I}_{i-1},\alpha_i,\mathbb{I}_{i+1},\ldots,\mathbb{I}_n)$. Simple verifications show that the map $\pi^i\colon \mathscr{I\!O}\!_{\infty}(\mathbb{Z}^n_{\operatorname{lex}})\rightarrow \mathscr{I\!O}\!_{\infty}(\mathbb{Z}^n_{\operatorname{lex}})$ is a homomorphism. Let ${\pi^i}^\sharp$ be the congruence on the semigroup $\mathscr{I\!O}\!_{\infty}(\mathbb{Z}^n_{\operatorname{lex}})$ which is generated by the homomorphism $\pi^i$.

Let $S$ be an inverse semigroup. For any congruence $\rho$ on $S$ we define a congruence $\rho_{\min}$ on $S$ as follows:
\begin{equation*}
    a\rho_{\min}b \qquad \hbox{if and only if} \qquad ae=be \quad \hbox{for some} \quad e\in E(S)  \quad \hbox{and} \quad e\rho a^{-1}a\rho b^{-1}b,
\end{equation*}
(see: \cite[Section~III.2]{Petrich1984}). Then Proposition~2.17 of \cite{GutikPozdnyakova2014} implies that
\begin{equation*}
{\pi^i}^\sharp_{\min}= \sigma_{[1]}\circ\ldots\circ\sigma_{[i-1]}\circ\sigma_{[i+1]}\circ\ldots \circ\sigma_{[n]}
\end{equation*}
for every $i=1,\ldots,n$.

This paper is a continuation of \cite{GutikPozdnyakova2014} and we study congruences on the semigroup $\mathscr{I\!O}\!_{\infty}(\mathbb{Z}^n_{\operatorname{lex}})$.  Here we describe the structure of the sublattice of congruences on $\mathscr{I\!O}\!_{\infty}(\mathbb{Z}^n_{\operatorname{lex}})$ which contained in the least group congruence.

For arbitrary elements $\alpha=(\alpha_1,\ldots,\alpha_n)$ and $\beta=(\beta_1,\ldots,\beta_n)$ of the semigroup $\mathscr{I\!O}\!_{\infty}(\mathbb{Z}^n_{\operatorname{lex}})$ we define:
\begin{equation*}
    \mathbf{D}_{\alpha,\beta}=\left\{i\in\{1,\ldots,n\}\mid \alpha_i\neq\beta_i \right\}.
\end{equation*}
It is obvious that elements $\alpha, \beta\in \mathscr{I\!O}\!_{\infty}(\mathbb{Z}^n_{\operatorname{lex}})$ are equal if and only if $\mathbf{D}_{\alpha,\beta}=\varnothing$.

\begin{lemma}\label{lemma-1}
Let $\mathfrak{C}$ be a congruence on the semigroup $\mathscr{I\!O}\!_{\infty}(\mathbb{Z}^n_{\operatorname{lex}})$. Let $\alpha$ and $\beta$ be two distinct $\mathfrak{C}$-equivalent elements of the semigroup $\mathscr{I\!O}\!_{\infty}(\mathbb{Z}^n_{\operatorname{lex}})$. Then there exists an element $\varpi$ in $\mathscr{I\!O}\!_{\infty}(\mathbb{Z}^n_{\operatorname{lex}})$ such that $\mathbb{I}\mathfrak{C}\varpi$ and $\mathbf{D}_{\mathbb{I},\varpi}=\mathbf{D}_{\alpha,\beta}$.
\end{lemma}

\begin{proof}
By Proposition~2.3$(iv)$ from \cite{GutikPozdnyakova2014} the semigroup $\mathscr{I\!O}\!_{\infty}(\mathbb{Z}^n_{\operatorname{lex}})$ is isomorphic to the direct power $\left(\mathscr{I\!O}\!_{\infty}(\mathbb{Z})\right)^{n}$. We denote $\alpha=(\alpha_1,\ldots,\alpha_n)$ and $\beta=(\beta_1,\ldots,\beta_n)$. Then for every $i\in \mathbf{D}_{\alpha,\beta}$ we have that $\alpha_i\neq\beta_i$.

We fix an arbitrary $i\in \mathbf{D}_{\alpha,\beta}$. Then one of the following cases holds:
\begin{itemize}
  \item[1)] $\alpha_i\mathscr{H}\beta_i$ in $\mathscr{I\!O}\!_{\infty}(\mathbb{Z})$;
  \item[2)] $\alpha_i$ and $\beta_i$ are not $\mathscr{H}$-equivalent in $\mathscr{I\!O}\!_{\infty}(\mathbb{Z})$.
\end{itemize}

Suppose that case 1) holds. By Proposition~2.3 of \cite{GutikRepovs2012} the semigroup $\mathscr{I\!O}\!_{\infty}(\mathbb{Z})$ is bisimple and hence by Theorem~2.3 from \cite{CP} there exist $\gamma_i,\delta_i\in \mathscr{I\!O}\!_{\infty}(\mathbb{Z})$ such that $\eta_i=\gamma_i\alpha_i\delta_i$ and $\zeta_i=\gamma_i\beta_i\delta_i$ are distinct elements of the group of units of the semigroup $\mathscr{I\!O}\!_{\infty}(\mathbb{Z})$. Then we have that $\eta_i^{-1}\eta_i=\eta_i^{-1}\gamma_i\alpha_i\delta_i=\mathbb{I}_i$ is the unit of the semigroup $\mathscr{I\!O}\!_{\infty}(\mathbb{Z})$ and $\eta_i^{-1}\zeta_i=\eta_i^{-1}\gamma_i\beta_i\delta_i\neq\mathbb{I}_i$. Hence, without loss of generality we can assume that there exist elements $\gamma_i$ and $\delta_i$ of the semigroup $\mathscr{I\!O}\!_{\infty}(\mathbb{Z})$ such that $\gamma_i\alpha_i\delta_i=\mathbb{I}_i$ is the unit of $\mathscr{I\!O}\!_{\infty}(\mathbb{Z})$ and $\gamma_i\beta_i\delta_i\neq\mathbb{I}_i$.

Suppose that the elements $\alpha_i$ and $\beta_i$ are not $\mathscr{H}$-equivalent in $\mathscr{I\!O}\!_{\infty}(\mathbb{Z})$. Then by Proposition~2.1$(vii)$ of \cite{GutikRepovs2012} we have that at least one of the following conditions holds:
\begin{equation*}
    \operatorname{dom}\alpha_i\neq \operatorname{dom}\beta_i \qquad \hbox{or} \qquad \operatorname{ran}\alpha_i\neq \operatorname{ran}\beta_i.
\end{equation*}
Since every subset with finite complement in $\mathbb{Z}$ is order isomorphic to $\mathbb{Z}$ we conclude that there exist monotone bijective maps $\gamma_i\colon\mathbb{Z}\rightarrow\operatorname{dom}\alpha_i$ and $\delta_i\colon\operatorname{ran}\alpha_i\rightarrow\mathbb{Z}$. Then we have that $\gamma_i\alpha_i\delta_i$ is an element of the group of units of the semigroup $\mathscr{I\!O}\!_{\infty}(\mathbb{Z})$, because $\operatorname{dom}(\gamma_i\alpha_i\delta_i)= \operatorname{ran}(\gamma_i\alpha_i\delta_i)=\mathbb{Z}$.

Suppose we have that $\operatorname{dom}\alpha_i\neq \operatorname{dom}\beta_i$. If there exists an integer $k\in\operatorname{dom}\alpha_i$ such that $k\notin\operatorname{dom}\beta_i$, then $(k)\gamma_i^{-1}\in\operatorname{dom}(\gamma_i\alpha_i\delta_i)$ and $(k)\gamma_i^{-1}\notin\operatorname{dom}(\gamma_i\beta_i\delta_i)$. If there exists an integer $k\in\operatorname{dom}\beta_i$ such that $k\notin\operatorname{dom}\alpha_i$, then $(k)\gamma_i^{-1}\in\operatorname{dom}(\gamma_i\beta_i\delta_i)$ and $(k)\gamma_i^{-1}\notin\operatorname{dom}(\gamma_i\alpha_i\delta_i)$. Therefore, we get that $\operatorname{dom}(\gamma_i\beta_i\delta_i)\neq \operatorname{dom}(\gamma_i\alpha_i\delta_i)$.

Suppose we have that $\operatorname{ran}\alpha_i\neq \operatorname{ran}\beta_i$. If there exists an integer $k\in\operatorname{ran}\alpha_i$ such that $k\notin\operatorname{ran}\beta_i$, then $(k)\delta_i\in\operatorname{ran}(\gamma_i\alpha_i\delta_i)$ and $(k)\delta_i\notin\operatorname{ran}(\gamma_i\beta_i\delta_i)$. If there exists an integer $k\in\operatorname{ran}\beta_i$ such that $k\notin\operatorname{ran}\alpha_i$, then $(k)\delta_i\in\operatorname{ran}(\gamma_i\beta_i\delta_i)$ and $(k)\delta_i\notin\operatorname{ran}(\gamma_i\alpha_i\delta_i)$. This implies that  $\operatorname{ran}(\gamma_i\beta_i\delta_i)\neq \operatorname{ran}(\gamma_i\alpha_i\delta_i)$.

Since every translation on an arbitrary element of the group of units of the semigroup $\mathscr{I\!O}\!_{\infty}(\mathbb{Z})$ is a bijective map of the set of integers $\mathbb{Z}$, without loss of generality we can assume that the element $\gamma_i\alpha_i\delta_i$ is the unit of the semigroup $\mathscr{I\!O}\!_{\infty}(\mathbb{Z})$.

Next, we define elements $\gamma=(\gamma_1,\ldots,\gamma_n)$ and $\delta=(\delta_1,\ldots,\delta_n)$ of the semigroup $\mathscr{I\!O}\!_{\infty}(\mathbb{Z}^n_{\operatorname{lex}})$ in the following way. For $i\in \mathbf{D}_{\alpha,\beta}$ we define $\gamma_i$ and $\delta_i$ to be the elements of the semigroup $\mathscr{I\!O}\!_{\infty}(\mathbb{Z})$ so constructed above. For $i\in\{1,\ldots,n\}\setminus \mathbf{D}_{\alpha,\beta}$ we put $\gamma_i$ and $\delta_i$ are the elements of the semigroup $\mathscr{I\!O}\!_{\infty}(\mathbb{Z})$ such that $\gamma_i\alpha_i\delta_i= \gamma_i\beta_i\delta_i=\mathbb{I}_i$ is the unit of the semigroup $\mathscr{I\!O}\!_{\infty}(\mathbb{Z})$. The existence of so elements $\gamma_i$ and $\delta_i$ in $\mathscr{I\!O}\!_{\infty}(\mathbb{Z})$ follows from Theorem~2.3 of \cite{CP} and the fact that the semigroup $\mathscr{I\!O}\!_{\infty}(\mathbb{Z})$ is bisimple (see \cite[Proposition~2.3]{GutikRepovs2012}).

Hence we get that
\begin{equation*}
    \gamma\alpha\delta=\mathbb{I}, \quad \varpi=\gamma\beta\delta\neq\mathbb{I} \quad \hbox{and} \quad \varpi\mathfrak{C}\mathbb{I} \quad \hbox{in} \quad \mathscr{I\!O}\!_{\infty}(\mathbb{Z}^n_{\operatorname{lex}}).
\end{equation*}
Moreover, our construction implies that $\mathbf{D}_{\mathbb{I},\varpi}=\mathbf{D}_{\alpha,\beta}$.
\end{proof}

\begin{lemma}\label{lemma-2}
Let $\mathfrak{C}$ be a congruence on the semigroup $\mathscr{I\!O}\!_{\infty}(\mathbb{Z}^n_{\operatorname{lex}})$. Let $\alpha$ and $\beta$ be two distinct $\mathfrak{C}$-equivalent elements of the semigroup $\mathscr{I\!O}\!_{\infty}(\mathbb{Z}^n_{\operatorname{lex}})$. Then there exists an element $\psi$ in $\mathscr{I\!O}\!_{\infty}(\mathbb{Z}^n_{\operatorname{lex}})$ such that $\mathbb{I}\mathfrak{C}\psi$, $\mathbf{D}_{\mathbb{I},\psi}=\mathbf{D}_{\alpha,\beta}$ and elements $\mathbb{I}$ and $\psi$ are not $\mathscr{H}$-equivalent in $\mathscr{I\!O}\!_{\infty}(\mathbb{Z}^n_{\operatorname{lex}})$.
\end{lemma}

\begin{proof}
If $\alpha$ and $\beta$ are not $\mathscr{H}$-equivalent elements of the semigroup $\mathscr{I\!O}\!_{\infty}(\mathbb{Z}^n_{\operatorname{lex}})$, then by case 2) of the proof of Lemma~\ref{lemma-1} we obtain that $\mathbb{I}\mathfrak{C}\varpi=\gamma\beta\delta$ and the elements $\mathbb{I}$ and $\varpi$ are not $\mathscr{H}$-equivalent in $\mathscr{I\!O}\!_{\infty}(\mathbb{Z}^n_{\operatorname{lex}})$.

Next, we suppose that $\alpha\mathscr{H}\beta$ and put $\alpha=(\alpha_1,\ldots,\alpha_n)$ and $\beta=(\beta_1,\ldots,\beta_n)$. Then by Proposition~2.3 of \cite{GutikRepovs2012} the semigroup $\mathscr{I\!O}\!_{\infty}(\mathbb{Z})$ is bisimple and hence by Theorem~2.3 of \cite{CP} for every $i=1,\ldots,n$ there exist $\gamma_i,\delta_i\in \mathscr{I\!O}\!_{\infty}(\mathbb{Z})$ such that $\gamma_i\alpha_i\delta_i=\mathbb{I}_i$ is the unit of the semigroup $\mathscr{I\!O}\!_{\infty}(\mathbb{Z})$ and $\gamma_i\beta_i\delta_i\neq\mathbb{I}_i$ for each $i\in \mathbf{D}_{\alpha,\beta}$. Since $\alpha\mathscr{H}\beta$ and by Proposition~2.3$(v)$ of \cite{GutikPozdnyakova2014} the semigroup $\mathscr{I\!O}\!_{\infty}(\mathbb{Z}^n_{\operatorname{lex}})$ is isomorphic to the direct power $\left(\mathscr{I\!O}\!_{\infty}(\mathbb{Z})\right)^n$ we conclude that $\gamma_i\beta_i\delta_i$ is an element of the group of units of $\mathscr{I\!O}\!_{\infty}(\mathbb{Z})$ for each $i\in\{1,\ldots,n\}$, and moreover $\gamma_i\beta_i\delta_i=\mathbb{I}_i=\gamma_i\alpha_i\delta_i$ for any $i\in\{1,\ldots,n\}\setminus\mathbf{D}_{\alpha,\beta}$.

We denote $\gamma=(\gamma_1,\ldots,\gamma_n)$ and $\delta=(\delta_1,\ldots,\delta_n)$ and put $\kappa=(\kappa_1,\ldots,\kappa_n)=\gamma\beta\delta$. Then we have that $\mathbf{D}_{\alpha,\beta}=\mathbf{D}_{\mathbb{I},\kappa}$. Also the relation $\alpha\mathscr{H}\beta$ implies that $\mathbb{I}\mathscr{H}\kappa$, and since $\mathscr{I\!O}\!_{\infty}(\mathbb{Z}^n_{\operatorname{lex}})$ is an inverse semigroup we get that $\mathbb{I}\mathscr{H}\kappa^m$ for every integer $m$.  By Proposition~2.2 of \cite{GutikRepovs2012} the group of units of the semigroup $\mathscr{I\!O}\!_{\infty}(\mathbb{Z})$ is isomorphic to $\mathbb{Z}(+)$. Hence, this implies that without loss of generality we can assume that $(p)\kappa_i=p+m_i$, where $m_i\neq 0$, for every $i\in \mathbf{D}_{\alpha,\beta}$ because.

Next, for every integer $i=1,\ldots,n$ we define a partial map $\chi_i\colon \mathbb{Z}\rightharpoonup\mathbb{Z}$ in the following way:
\begin{itemize}
  \item[(a)] if $i\in\{1,\ldots,n\}\setminus\mathbf{D}_{\alpha,\beta}$, then we define $\chi_i\colon \mathbb{Z}\rightarrow \mathbb{Z}$ be the identity map;

  \item[(b)] if $i\in\mathbf{D}_{\alpha,\beta}$ and $m_i\geqslant 1$, then we define $\operatorname{dom}\chi_i=\mathbb{Z}$, $\operatorname{ran}\chi_i=\mathbb{Z}\setminus\{1,\ldots,m_i\}$ and
\begin{equation*}
(k)\chi_i=
\left\{
  \begin{array}{cl}
    k+m_i, & \hbox{if~}\, k\geqslant 1;\\
    k,     & \hbox{if~}\, k\leqslant 0;
  \end{array}
\right.
\end{equation*}

  \item[(c)] if $i\in\mathbf{D}_{\alpha,\beta}$ and $m_i\leqslant -1$, then we define $\operatorname{dom}\chi_i=\mathbb{Z}$, $\operatorname{ran}\chi_i=\mathbb{Z}\setminus\{m_i,\ldots,-1\}$ and
\begin{equation*}
(k)\chi_i=
\left\{
  \begin{array}{cl}
    k,     & \hbox{if~}\, k\geqslant 0;\\
    k+m_i, & \hbox{if~}\, k\leqslant -1.
  \end{array}
\right.
\end{equation*}
\end{itemize}
We put $\chi=(\chi_1,\ldots,\chi_n)$. The definition of the semigroup $\mathscr{I\!O}\!_{\infty}(\mathbb{Z}^n_{\operatorname{lex}})$ implies that $\chi$ and its inverse $\chi^{-1}$ are elements of $\mathscr{I\!O}\!_{\infty}(\mathbb{Z}^n_{\operatorname{lex}})$. Simple verifications show that $\mathbb{I}=\chi\chi^{-1}=\chi\mathbb{I}\chi^{-1}$. Also, since $\mathfrak{C}$ is a congruence on the semigroup $\mathscr{I\!O}\!_{\infty}(\mathbb{Z}^n_{\operatorname{lex}})$ we conclude that $\mathbb{I}=\chi\mathbb{I}\chi^{-1}\mathfrak{C}\chi\kappa\chi^{-1}$.

Now simple calculations imply that
\begin{itemize}
  \item[$(i)$] if $m_i>0$ then
\begin{equation*}
(k)\chi_i\kappa_i\chi_i^{-1}= \left\{
  \begin{array}{cl}
    k+m_i,            & \hbox{if~}\, k\geqslant 1;\\
    \hbox{undefined}, & \hbox{if~}\, -m_i<k\leqslant 0;\\
    k+m_i,            & \hbox{if~}\, k\leqslant-m_i,
  \end{array}
\right.
\end{equation*}
\end{itemize}
and similarly
\begin{itemize}
  \item[$(ii)$] if $m_i<0$ then
\begin{equation*}
(k)\chi_i\kappa_i\chi_i^{-1}= \left\{
  \begin{array}{cl}
    k+m_i,            & \hbox{if~}\, k\geqslant-m_i;\\
    \hbox{undefined}, & \hbox{if~}\, 0\leqslant k<-m_i;\\
    k+m_i,            & \hbox{if~}\, k\leqslant-1.
  \end{array}
\right.
\end{equation*}
\end{itemize}

Next we put $\psi=\chi\kappa\chi^{-1}$, and hence we obtain that $\mathbb{I}\mathfrak{C}\psi$ but $\operatorname{dom}\psi\neq\mathbb{Z}$. This completes the proof of our lemma.
\end{proof}

\begin{remark}\label{remark-3}
The proof of Lemma~\ref{lemma-2} implies that for element $\psi=(\psi_1,\ldots,\psi_n)$ the following property holds:
\begin{quote}
    $\psi_i$ is not $\mathscr{H}$-equivalent to the unit of the semigroup $\mathscr{I\!O}\!_{\infty}(\mathbb{Z})$ for every $i\in\mathbf{D}_{\alpha,\beta}$.
\end{quote}
\end{remark}

\begin{proposition}\label{proposition-4}
Let $\mathfrak{C}$ be a congruence on the semigroup $\mathscr{I\!O}\!_{\infty}(\mathbb{Z}^n_{\operatorname{lex}})$. Let $\alpha$ and $\beta$ be two distinct $\mathfrak{C}$-equivalent elements of the semigroup $\mathscr{I\!O}\!_{\infty}(\mathbb{Z}^n_{\operatorname{lex}})$. Then there exists a non-unit idempotent $\varepsilon$ in $\mathscr{I\!O}\!_{\infty}(\mathbb{Z}^n_{\operatorname{lex}})$ such that $\mathbb{I}\mathfrak{C}\varepsilon$ and $\mathbf{D}_{\mathbb{I},\varepsilon}=\mathbf{D}_{\alpha,\beta}$.
\end{proposition}

\begin{proof}
Lemma~\ref{lemma-2} implies that there exists an element $\psi$ of the semigroup $\mathscr{I\!O}\!_{\infty}(\mathbb{Z}^n_{\operatorname{lex}})$ such that $\psi\mathfrak{C}\mathbb{I}$, $\mathbf{D}_{\mathbb{I},\psi}= \mathbf{D}_{\alpha,\beta}$ and elements $\mathbb{I}$ and $\psi$ are not $\mathscr{H}$-equivalent in $\mathscr{I\!O}\!_{\infty}(\mathbb{Z}^n_{\operatorname{lex}})$. Also, by Remark~\ref{remark-3} for every integer $i\in\mathbf{D}_{\alpha,\beta}$ the element $\psi_i$ is not $\mathscr{H}$-equivalent to the unit $\mathbb{I}_i$ of the semigroup $\mathscr{I\!O}\!_{\infty}(\mathbb{Z})$. This implies that for every integer $i\in\mathbf{D}_{\alpha,\beta}$ at least one of the following conditions holds:
\begin{equation*}
    \psi_i\psi_i^{-1}\neq\mathbb{I}_i \qquad \hbox{or} \qquad \psi_i^{-1}\psi_i\neq\mathbb{I}_i \qquad \hbox{in} \qquad \mathscr{I\!O}\!_{\infty}(\mathbb{Z}).
\end{equation*}

Since $\mathscr{I\!O}\!_{\infty}(\mathbb{Z}^n_{\operatorname{lex}})$ is an inverse semigroup we have that $\mathbb{I}\mathfrak{C}\psi^{-1}$. This implies that $\mathbb{I}\mathfrak{C}\psi\psi^{-1}$ and $\mathbb{I}\mathfrak{C}\psi^{-1}\psi$, and hence we get that $\mathbb{I}\mathfrak{C}\varepsilon$, where $\varepsilon=\psi\psi^{-1}\psi^{-1}\psi$. The above arguments show that $\mathbf{D}_{\mathbb{I},\varepsilon}=\mathbf{D}_{\alpha,\beta}$.
\end{proof}

\begin{proposition}\label{proposition-5}
Let $\mathfrak{C}$ be a congruence on the semigroup $\mathscr{I\!O}\!_{\infty}(\mathbb{Z}^n_{\operatorname{lex}})$. Let $\alpha$ and $\beta$ be two distinct $\mathfrak{C}$-equivalent elements of the semigroup $\mathscr{I\!O}\!_{\infty}(\mathbb{Z}^n_{\operatorname{lex}})$. Then $\mathbb{I}\mathfrak{C}\varepsilon$ for any idempotent $\varepsilon$ in $\mathscr{I\!O}\!_{\infty}(\mathbb{Z}^n_{\operatorname{lex}})$ such that   $\mathbf{D}_{\mathbb{I},\varepsilon}=\mathbf{D}_{\alpha,\beta}$.
\end{proposition}

\begin{proof}
By Proposition~\ref{proposition-4} there exists an idempotent $\varepsilon$ of the semigroup $\mathscr{I\!O}\!_{\infty}(\mathbb{Z}^n_{\operatorname{lex}})$ such that $\mathbb{I}\mathfrak{C}\varepsilon$ and $\mathbf{D}_{\mathbb{I},\varepsilon}=\mathbf{D}_{\alpha,\beta}$. We fix an arbitrary non-unit idempotent $\iota\in \mathscr{I\!O}\!_{\infty}(\mathbb{Z}^n_{\operatorname{lex}})$ such that $\varepsilon\leqslant\iota$ in $E(\mathscr{I\!O}\!_{\infty}(\mathbb{Z}^n_{\operatorname{lex}}))$. Then we have that $\iota\mathbb{I}=\iota$ and hence the relation $\mathbb{I}\mathfrak{C}\varepsilon$ implies that $\iota=\iota\mathbb{I} \mathfrak{C}\iota\varepsilon=\varepsilon\mathfrak{C}\mathbb{I}$. Therefore, for every $i\in\mathbf{D}_{\alpha,\beta}$ there exists an idempotent $\varepsilon_i^\circ$ such that $\varepsilon_i^\circ\mathfrak{C}\mathbb{I}$ and the set $\mathbb{Z}\setminus\operatorname{dom}\varepsilon_i^\circ$ is singleton. We put $\{m_i\}=\mathbb{Z}\setminus\operatorname{dom}\varepsilon_i^\circ$ for every integer $i\in\mathbf{D}_{\alpha,\beta}$. We fix an arbitrary integer $p_i$ for $i\in\mathbf{D}_{\alpha,\beta}$ and define the map $\varrho_i\colon\mathbb{Z}\rightarrow\mathbb{Z}$ by the formula:
\begin{equation*}
    (j)\varrho_i=j-m_i+p_i, \qquad \hbox{for every}\quad j\in\mathbb{Z}.
\end{equation*}
Then $\varrho_i$ is an element of the group of units of the semigroup $\mathscr{I\!O}\!_{\infty}(\mathbb{Z})$ and hence $\varrho_i\varrho_i^{-1}=\varrho_i^{-1}\varrho_i=\mathbb{I}_i$ in $\mathscr{I\!O}\!_{\infty}(\mathbb{Z})$. Moreover, it is obvious that $\varrho_i^{-1}\varepsilon_i^\circ\varrho_i$ is an idempotent of the semigroup $\mathscr{I\!O}\!_{\infty}(\mathbb{Z})$ such that $\operatorname{dom}(\varrho_i^{-1}\varepsilon_i^\circ\varrho_i)=\mathbb{Z}\setminus \{p_i\}$. Also, we obtained that $\mathbb{I}_i=\varrho_i^{-1}\mathbb{I}_i\varrho_i \mathfrak{C}\varrho_i^{-1}\varepsilon_i^\circ\varrho_i$ in $\mathscr{I\!O}\!_{\infty}(\mathbb{Z})$. Now the definition of the semigroup $\mathscr{I\!O}\!_{\infty}(\mathbb{Z})$ implies that $\mathbb{I}\mathfrak{C}\iota_i^\circ$ for any idempotent $\iota$ in $\mathscr{I\!O}\!_{\infty}(\mathbb{Z})$, because every idempotent $\iota$ in the semigroup $\mathscr{I\!O}\!_{\infty}(\mathbb{Z})$ is equal to a product of finitely many idempotents of the form $\iota_i^\circ$, $i\in\{1,\ldots,n\}$, with the property that the set $\mathbb{Z}\setminus\operatorname{dom}\iota_i^\circ$ is singleton. Then for every idempotent $\varepsilon$ of the semigroup $\mathscr{I\!O}\!_{\infty}(\mathbb{Z}^n_{\operatorname{lex}})$ with the property $\mathbf{D}_{\mathbb{I},\varepsilon}=\mathbf{D}_{\alpha,\beta}$ we have that
\begin{equation*}
    \varepsilon=\varepsilon_{i_1}^\circ\cdot\ldots\cdot\varepsilon_{i_k}^\circ, \qquad \hbox{where} \quad \left\{i_1,\ldots,i_k\right\}=\mathbf{D}_{\alpha,\beta},
\end{equation*}
and hence $\mathbb{I}\mathfrak{C}\varepsilon$. This completes the proof of the proposition.
\end{proof}

\begin{theorem}\label{theorem-6}
Let $\mathfrak{C}$ be a congruence on the semigroup $\mathscr{I\!O}\!_{\infty}(\mathbb{Z}^n_{\operatorname{lex}})$. Then the following statements hold:
\begin{itemize}
  \item[$(i)$] If $\Delta_{\mathscr{I\!O}\!_{\infty}(\mathbb{Z}^n_{\operatorname{lex}})}\subseteq \mathfrak{C}\subseteq\sigma_{[i_m]}$ for some $i_m\in\{1,\ldots,n\}$, then either $\Delta_{\mathscr{I\!O}\!_{\infty}(\mathbb{Z}^n_{\operatorname{lex}})}= \mathfrak{C}$ or $\mathfrak{C}=\sigma_{[i_m]}$.
  \item[$(ii)$] If $\sigma_{[i_1,\ldots,i_m]}\subseteq \mathfrak{C}\subseteq\sigma_{[i_1,\ldots,i_m,i_{m+1}]}$, for any subset $\{i_1,\ldots,i_m,i_{m+1}\}\subseteq\{1,\ldots,n\}$, then either $\sigma_{[i_1,\ldots,i_m]}= \mathfrak{C}$ or $\mathfrak{C}=\sigma_{[i_1,\ldots,i_m,i_{m+1}]}$.
\end{itemize}
\end{theorem}

\begin{proof}
By Proposition~2.15 from \cite{GutikPozdnyakova2014} we have that for any collection $\{i_1,\ldots,i_k\}\subseteq\{1,\ldots,n\}$ of distinct indices, $k\le n$, and, hence, $\alpha\sigma_{[i_1,\ldots,i_k]}\beta$ in $\mathscr{I\!O}\!_{\infty}(\mathbb{Z}^n_{\operatorname{lex}})$ if and only if $\alpha\varepsilon_{i_1}^\circ\ldots\varepsilon_{i_k}^\circ= \beta\varepsilon_{i_1}^\circ\ldots\varepsilon_{i_k}^\circ$ for some idempotents $\varepsilon_{i_1}^\circ,\ldots,\varepsilon_{i_k}^\circ\in \mathscr{I\!O}\!_{\infty}(\mathbb{Z}^n_{\operatorname{lex}})$. This implies that $\mathbb{I}\sigma_{[i_1,\ldots,i_k]}\varepsilon$ for every idempotent $\varepsilon$ of the semigroup $\mathscr{I\!O}\!_{\infty}(\mathbb{Z}^n_{\operatorname{lex}})$ such that $\mathbf{D}_{\mathbb{I},\varepsilon}\subseteq\{i_1,\ldots,i_k\}$. Then applying Proposition~\ref{proposition-4} we get the statement of the theorem.
\end{proof}

For any proper subset if indices $\mathrm{I}\subset\{1,\ldots,n\}$ we define a map $\pi^{\mathrm{I}}\colon \mathscr{I\!O}\!_{\infty}(\mathbb{Z}^n_{\operatorname{lex}})\rightarrow \mathscr{I\!O}\!_{\infty}(\mathbb{Z}^n_{\operatorname{lex}})$ by the formula $(\alpha_1,\ldots,\alpha_n)\pi_{\mathrm{I}}= (\beta_1,\ldots,\beta_n)$, where
\begin{equation*}
    \beta_i=
\left\{
  \begin{array}{ll}
    \alpha_i, & \hbox{if~} i\in \mathrm{I};\\
    \mathbb{I}_i, & \hbox{if~} i\in\{1,\ldots,n\}\setminus \mathrm{I}.
  \end{array}
\right.
\end{equation*}
 Simple verifications show that such defined map $\pi^{\mathrm{I}}\colon \mathscr{I\!O}\!_{\infty}(\mathbb{Z}^n_{\operatorname{lex}})\rightarrow \mathscr{I\!O}\!_{\infty}(\mathbb{Z}^n_{\operatorname{lex}})$ is a homomorphism. Let ${\pi^{\mathrm{I}}}^\sharp$ be the congruence on $\mathscr{I\!O}\!_{\infty}(\mathbb{Z}^n_{\operatorname{lex}})$ which is generated by the homomorphism $\pi^{\mathrm{I}}$.

\begin{proposition}\label{proposition-7}
Let $\mathrm{I}$ be an arbitrary proper subset of $\{1,\ldots,n\}$. Then
${\pi^{\mathrm{I}}}^\sharp_{\min}= \sigma_{[i_1]}\circ\ldots\circ\sigma_{[i_k]}$, where $\{i_1,\ldots,i_k\}=\{1,\ldots,n\}\setminus\mathrm{I}$.
\end{proposition}

\begin{proof}
Suppose that $\alpha(\sigma_{[i_1]}\circ\ldots\circ \sigma_{[i_k]}) \beta$ in $\mathscr{I\!O}\!_{\infty}(\mathbb{Z}^n_{\operatorname{lex}})$ for some elements $\alpha=(\alpha_1,\ldots,\alpha_n)$ and $\beta=(\beta_1,\ldots,\beta_n)$.  Proposition~2.15 of \cite{GutikPozdnyakova2014} implies that $\alpha\varepsilon_{i_1}^\circ\ldots\varepsilon_{i_k}^\circ= \beta\varepsilon_{i_1}^\circ\ldots\varepsilon_{i_k}^\circ$ for some idempotent $\varepsilon=(\varepsilon_{1},\ldots,\varepsilon_{n})$ such that
$\varepsilon_{i}=\mathbb{I}_i$ for all $i\in\mathrm{I}$,
i.e., $\alpha\varepsilon=\beta\varepsilon$. Then we have that $\alpha_i=\beta_i$ for all $i\in\mathrm{I}$, and hence $\alpha\varepsilon^*=\beta\varepsilon^*$ for $\varepsilon^*=(\varepsilon^*_{1},\ldots,\varepsilon^*_{n})$, where
\begin{equation*}
    \varepsilon^*_{i}=
\left\{
  \begin{array}{cl}
    \alpha_i^{-1}\alpha_i=\beta_i^{-1}\beta_i, & \hbox{if~} i\in \mathrm{I};\\
    \varepsilon_{i}, & \hbox{if~} i\in\{1,\ldots,n\}\setminus \mathrm{I}.
  \end{array}
\right.
\end{equation*}
It is obvious that $\varepsilon^*{\pi^\mathrm{I}}^\sharp\alpha^{-1}\alpha{\pi^\mathrm{I}}^\sharp\beta^{-1}\beta$. This implies the inclusion $\sigma_{[i_1]}\circ\ldots\circ\sigma_{[i_k]} \subseteq{\pi^\mathrm{I}}^\sharp_{\min}$.

Suppose that $\alpha{\pi^\mathrm{I}}^\sharp_{\min}\beta$ in $\mathscr{I\!O}\!_{\infty}(\mathbb{Z}^n_{\operatorname{lex}})$ for some elements $\alpha=(\alpha_1,\ldots,\alpha_n)$ and $\beta=(\beta_1,\ldots,\beta_n)$. Then there exists an idempotent $\varepsilon=(\varepsilon_{1},\ldots,\varepsilon_{n})$ in $\mathscr{I\!O}\!_{\infty}(\mathbb{Z}^n_{\operatorname{lex}})$ such that $\alpha\varepsilon=\beta\varepsilon$ and $\varepsilon{\pi^\mathrm{I}}^\sharp \alpha^{-1}\alpha{\pi^\mathrm{I}}^\sharp\beta^{-1}\beta$. The last two equalities imply that $\alpha_i^{-1}\alpha_i=\beta_i^{-1}\beta_i=\varepsilon_i$ for all $i\in\mathrm{I}$. This and the equality $\alpha\varepsilon=\beta\varepsilon$ imply that $\alpha_i\varepsilon_i= \beta_i\varepsilon_i$  for all $i\in\mathrm{I}$ and hence we obtain that $\alpha_i=\alpha_i\alpha_i^{-1}\alpha_i= \alpha_i\varepsilon_i= \beta_i\varepsilon_i= \beta_i\beta_i^{-1}\beta_i=\beta_i$ for all $i\in\mathrm{I}$. Therefore we have that $\alpha\varepsilon^*=\beta\varepsilon^*$, where the idempotent $\varepsilon^*=(\varepsilon^*_{1},\ldots,\varepsilon^*_{n})$ defined in the following way
\begin{equation*}
    \varepsilon^*_{i}=
\left\{
  \begin{array}{cl}
    \alpha_i^{-1}\alpha_i, & \hbox{if~} i\in \mathrm{I};\\
    \varepsilon_{i}, & \hbox{if~} i\in\{1,\ldots,n\}\setminus \mathrm{I}.
  \end{array}
\right.
\end{equation*}
This implies that $\alpha\varepsilon_{i_1}^\circ\ldots\varepsilon_{i_k}^\circ= \beta\varepsilon_{i_1}^\circ\ldots\varepsilon_{i_k}^\circ$. By Proposition~2.15 of \cite{GutikPozdnyakova2014} we get that  $\alpha(\sigma_{[i_1]}\circ\ldots\circ \sigma_{[i_k]}) \beta$ in $\mathscr{I\!O}\!_{\infty}(\mathbb{Z}^n_{\operatorname{lex}})$, and hence we get that ${\pi^i}^\sharp_{\min}\subseteq \sigma_{[i_1]}\circ\ldots\circ \sigma_{[i_k]}$. This completes the proof of equality ${\pi^\mathrm{I}}^\sharp_{\min}= \sigma_{[i_1]}\circ\ldots\circ \sigma_{[i_k]}$.
\end{proof}

\section*{Acknowledgements}

The authors are grateful to the referee for several useful comments
and suggestions.


\end{document}